\documentclass[12pt]{amsart}
\usepackage{amsmath, enumerate}
\usepackage{amsfonts, amssymb, amsthm, xcolor, stmaryrd, dsfont}

\numberwithin{equation}{section}
\theoremstyle{plain}

\newtheorem{thm}{Theorem}[section]

\newtheorem{prop}[thm]{Proposition}

\theoremstyle{definition}

\theoremstyle{remark}
\newtheorem{rmk}[thm]{Remark}

\newcommand{\Hc}{{\mathcal{H}}}
\newcommand{\lp}{\left (}
\newcommand{\rp}{\right )}

\newcommand{\hx}{{x}}
\newcommand{\hy}{{y}}
\newcommand{\erf}{\mathrm{erf}}
\newcommand{\erfc}{\mathrm{erfc}}
\newcommand{\Dh}{\mathcal{D}^+}
\newcommand{\Dn}{\mathcal{D}^*}
\newcommand{\fn}{f^*}
\newcommand{\tTheta}{\tilde{\Theta}}
\newcommand{\tvartheta}{\tilde{\vartheta}}
\newcommand{\tc}{\tilde{c}}
\newcommand{\tvarphi}{\tilde{\varphi}}
\newcommand{\SL}{{\mathrm{SL}}}
\newcommand{\Nb}{\mathbb{N}}
\newcommand{\Zb}{\mathbb{Z}}
\newcommand{\Qb}{\mathbb{Q}}
\newcommand{\Cb}{\mathbb{C}}
\newcommand{\ef}{\mathfrak{e}}
\newcommand{\ebf}{\mathbf{e}}
\newcommand{\Rb}{\mathbb{R}}
\newcommand{\Dc}{{\mathcal{D}}}
\newcommand{\sgn}{\mathrm{sgn}}
\newcommand{\Fc}{{\mathcal{F}}}
\newcommand{\smat}[4]{\left(\begin{smallmatrix}
                 #1 & #2\\
                 #3 & #4
\end{smallmatrix}\right)}
\newcommand{\half}{{\tfrac{1}{2}}}
\newcommand{\varep}{\varepsilon}

\begin{document}
\title{Harmonic Eisenstein Series of Weight One}
\author{Yingkun Li}
\date{\today}
\thanks{This work was partially supported by the DFG grant BR-2163/4-2 and an NSF postdoctoral fellowship.}
\address{Fachbereich Mathematik,
Technische Universit\"at Darmstadt, Schlossgartenstrasse 7, D--64289
Darmstadt, Germany}
\email{li@mathematik.tu-darmstadt.de}

\maketitle

\begin{abstract}
In this short note, we will construct a harmonic Eisenstein series of weight one, whose image under the $\xi$-operator is a weight one Eisenstein series studied by Hecke \cite{Hecke26}.
\end{abstract}

\section{Introduction}
In the theory of automorphic forms, Eisenstein series occupy an important place. Holomorphic Eisenstein series can be explicitly constructed and are usually the first examples of modular forms people encounter. Furthermore, their constant Fourier coefficients are special values of the Riemann zeta function, whereas the non-constant coefficients are the sums of the powers of divisors. Modularity then connects these two types of interesting quantities together.

Holomorphic theta series constructed from positive definite lattices provide another source of modular forms besides Eisenstein series. 
In \cite{Siegel51}, Siegel introduced non-holomorphic theta series associated to indefinite lattices, and showed that they can be integrated to produce Eisenstein series.
Later in his seminal work \cite{Weil65}, Weil studied this phenomenon for algebraic groups, and deduced the famous Siegel-Weil formula.

In the setting of theta correspondence between the orthogonal and sympletic groups, the Siegel-Weil formula is an equality between the integral of a theta function on the orthogonal side and an Eisenstein series on the symplectic side. 
With the knowledge of the theta kernel, one can then construct various symplectic Eisenstein series. A prototypical example of such a construction was already carried out by Hecke around 1926 \cite{Hecke26}, where he constructed a theta kernel $\Theta(\tau, t)$ from an indefinite lattice of signature $(1, 1)$ and integrated it to produce a holomorphic modular form $\vartheta(\tau)$ of weight one. This is an Eisenstein series if the lattice is isotropic and a cusp form otherwise.
In \cite{Kudla81}, Kudla extended this construction to produce holomorphic Siegel modular forms of genus $g$ and weight $\frac{g+1}{2}$, as a prelude to the important works by Kudla and Millson later \cite{KM86, KM87}.

In this note, we will consider a different theta kernel $\tilde{\Theta}(\tau, t)$ for an isotropic, indefinite lattice of signature $(1, 1)$.
Rather than holomorphic, its integral in $t$ is a harmonic function $\tilde{\vartheta}(\tau)$ and related to the holomorphic Eisenstein series $\vartheta(\tau)$ constructed by Hecke via
$$
\xi \tvartheta(\tau) = \vartheta(\tau),
$$
where $\xi = \xi_1$ is the differentiable operator introduced by Bruinier and Funke \cite{BF04}.
In the notion loc.\ cit., $\tvartheta(\tau)$ is a harmonic Maass form of weight one.
For any $k \in \half \Zb$, a harmonic Maass form of weight $k$ is a real analytic functions on the upper half-plane $\Hc := \{ \tau = u + i v: v > 0\}$ that transforms with weight $k$ with respect to a discrete subgroup of $\SL_2(\Rb)$, and is annihilated by the weight $k$ hyperbolic Laplacian 
$$
\Delta_k :=  - \xi_{2-k} \circ \xi_k, \quad
\xi_k(f) := 2i v^k \overline{\frac{\partial f}{\partial \overline{\tau}}}.
$$
Harmonic Maass forms can be written as the sum of a holomorphic part and a non-holomorphic part. 
The Fourier coefficients of their holomorphic parts are expected to contain interesting arithmetic information concerning the $\xi_k$-images of the non-holomorphic parts (see e.g.\ \cite{BO10, DIT11}).

In \cite{KRY99}, Kudla, Rapoport and Yang considered an Eisenstein series, which is harmonic. The Fourier coefficients of its holomorphic part are logarithms of rational numbers, and can be interpreted as the arithmetic degree of special divisors on an arithmetic curve. 
In view of their work and the Kudla program \cite{Kudla97}, we expect the Fourier coefficients of the harmonic Eisenstein series we construct to have a similar interpretation as well.

The idea to construct $\tvartheta(\tau)$ is rather straightforward. If we can construct a function $\tTheta(\tau, t)$ such that it is modular in $\tau$ and satisfies $\xi \tTheta(\tau, t) = \Theta(\tau, t)$ for each $t$, then simply integrating it in $t$ will produce the desirable $\tvartheta(\tau)$. This idea has already been used in \cite{BFI15}, where $\xi_{1/2}$ connected the theta kernels constructed from the Gaussian and the Kudla-Millson Schwartz form. 
In our setting, we will introduce an $L^\infty$ function $\tvarphi_\tau$, which is a $\xi$-preimage of the Schwartz function used in constructing $\Theta(\tau, t)$ under $\xi$ (see Prop.\ \ref{prop:tvarphi}). 
We will then use this function to form a theta kernel $\tTheta(\tau, t)$ and integrate it to obtain the harmonic Eisenstein series $\tvartheta(\tau)$ in Theorem \ref{thm:tvartheta} in the last section.

\subsection*{Acknowledgement.} The idea of the function $\tvarphi_\tau$ came out of discussions with Pierre Charollois during a visit to Universit\'{e} Paris 6 in November 2015. I am thankful for his encouragements that led to this note.

\section{Theta lift from $\mathrm{O}(1, 1)$ to $\SL_2$}
\label{sec:theta_lift}
In this section, we will recall the construction of the Eisenstein series in \cite{Hecke26} and \cite{Kudla81}. 
For $N \in \Nb$, let $L = N \Zb^2$ be a lattice with quadratic form $Q(\binom{a}{b}) := \frac{ab}{N}$ and $B(\cdot, \cdot) : L \times L \to \Zb$ the associated bilinear form. The dual lattice $L^* \subset V_\Qb := L \otimes \Qb$ is then $\Zb^2$ and the discriminant group is $L^*/L = (\Zb/N \Zb)^2$.

Let $\rho_L$ be the Weil representation of $\SL_2(\Zb)$ on $\Cb[L^*/L]$. 
As usual, let $\{\ef_h: h \in L^*/L \}$ denote the canonical basis of $\Cb[L^*/L]$ and $\ebf(a) := e^{2\pi i a}$ for any $a \in \Cb$.
Then the action of $\rho_L$ on the generators $T, S \in \SL_2(\Zb)$ is given by (see e.g.\ \cite[\textsection 4]{Borcherds98})
\begin{align*}
  \rho_L(T) (\ef_h) &= \ebf(Q(h)) \ef_h,  \;
\rho_L(S) (\ef_h) = \frac{1}{N} \sum_{\delta \in L^*/L} \ebf(-(\delta, h)) \ef_\delta.
\end{align*}
The symmetric domain attached to $V_\Rb := L \otimes \Rb$ is given by the hyperbola 
$$
\Dc := \left\{
Z \in V_\Rb | B(Z, Z) = -1 
\right\}.
$$
We denote one of its two connected components by $\Dc^+$ and parametrize it by
\begin{align*}
  \Phi: \Rb^\times_+ &\to \Dc^+ \\
t &\mapsto Z_t :=   \sqrt{\frac{N}{2}} \binom{t}{-1/t}.
\end{align*}
Let $W_t := \sqrt{N/2} \binom{ t}{1/ t} \in Z_t^\perp$.
Then $d \Phi \lp t \tfrac{d}{dt} \rp = W_t \in V_\Rb$ and $\{W_t, Z_t\}$ is an orthogonal basis of $V_\Rb$. 
For any $X = \binom{x_1}{x_2} \in V_\Rb$, one can write $X = X_{W_t} + X_{Z_t}$, where
$$
X_{W_t} := B(X, W_t) W_t = \frac{t^{-1} x_1 + t x_2}{\sqrt{2N}} W_t, \; X_{Z_t} := - B(X, Z_t) Z_t  =  \frac{t^{-1} x_1 - t x_2}{\sqrt{2N}}  Z_t.
$$
Then the majorant of $B(\cdot, \cdot)$ associated to $Z_t$, denoted by $B(\cdot, \cdot)_t$, is given by the positive definite quadratic form
$$
Q(X)_t := Q(X_{W_t}) - Q(X_{Z_t}) = \frac{B(X, W_t)^2 + B(X, Z_t)^2}{2} = \frac{t^{-2} x_1^2 + t^2 x_2^2}{2N}.
$$

Let $\Rb^{1, 1} = \{(x, y): x, y \in \Rb\}$ be a quadratic space of signature $(1, 1)$ with the quadratic form $Q'((x, y)) = \frac{x^2 - y^2}{2}$ with associated bilinear form $B'(\cdot, \cdot)$.
Given $\tau = u + iv \in \Hc$ in the upper half plane, we define the Schwartz function $\varphi_\tau$ on $\Rb^{1, 1}$ by
\begin{equation}
  \label{eq:varphi}
  \begin{split}
      \varphi_\tau: \Rb^{1, 1} &\to \Cb \\
(x, y) &\mapsto  \sqrt{2v} \cdot  x \cdot \ebf \lp \frac{x^2}{2} \tau - \frac{y^2}{2} \overline{\tau} \rp.
  \end{split}
\end{equation}
Now summing $\varphi_\tau$ over any even, integral lattice $M \subset \Rb^{1, 1}$ of rank 2 would produce a real-analytic theta series of weight 1 that transforms with respect to $\rho_M$.
In our setting, we let $M$ be the image of $L$ under the following isometry
\begin{equation}
  \label{eq:iotat}
  \begin{split}
      \iota_t: V_{\Rb} & \to \Rb^{1, 1} \\
X &\mapsto (B(X, W_t), B(X, Z_t))
  \end{split}
\end{equation}
for each $t \in \Rb^\times_+$.
Now the vector-valued theta function 
\begin{equation}
  \label{eq:Theta}
  \Theta(\tau, t) := \sum_{h \in L^*/L}  \Theta_{h}(\tau, t) \ef_h, \;
\Theta_{h}(\tau, t) := \sum_{X \in L + h} \varphi_{\tau}(\iota_t(X))
\end{equation}
transforms with weight 1 and representation $\rho_L$ in the variable $\tau$ by Theorem 4.1 in \cite{Borcherds98}.
For $h = \binom{h_1}{h_2} \in \Zb^2$, we have explicitly
$$
\Theta_h(\tau, t) = \sqrt{\frac{v}{N}} \sum_{\begin{subarray}{c} x_1 \equiv h_1 (N) \\ x_2 \equiv h_2 (N) \end{subarray}}
(t^{-1} x_1 + tx_2) \ebf \lp
\frac{x_1 x_2}{N} u + \frac{t^{-2} x_1^2 + t^2 x_2^2}{2N} iv
\rp
$$
Integrating over $t \in \Rb^\times_+$ with respect to the invariant differential $\frac{dt}{t}$ defines
\begin{equation}
\label{eq:vartheta}
\vartheta_h(\tau, s) := 
\int^1_0 \Theta_h(\tau, t) t^{s} \frac{dt}{t}  + \int^\infty_1 \Theta_h(\tau, t) t^{-s} \frac{dt}{t} .
\end{equation}
Here, the integral converges for $\Re s \gg 0$. As a function of $s$, it has analytic continuation to $s \in \Cb$. 
Let $\vartheta_h(\tau)$ be the constant term in the Laurent expansion of $\vartheta_h(\tau, s)$ around $s = 0$. 
Then $\vartheta_h(\tau)$ is holomorphic and $\vartheta(\tau) := \sum_{h \in L^*/L} \vartheta_h(\tau) \ef_h$ is an Eisenstein series of weight 1 on $\SL_2(\Zb)$ and transforms with respect to $\rho_L$. 
It has the following Fourier expansion (see \cite[Theorem 3.2]{Kudla81}).
\begin{prop}
  \label{prop:FE}
Write $h = \binom{h_1}{h_2} \in \Zb^2$. 
Then $\vartheta_h(\tau)$ has the Fourier expansion 
$\vartheta_h(\tau) = \sum_{n \in \Qb_{\ge 0}} c_h(n) q^n$ with
\begin{equation}
  \label{eq:FC}
  \begin{split}
c_h(0) &:= 
\begin{cases}
  \frac{1}{2} - \langle \frac{h_1}{N} \rangle & N \nmid h_1, n \mid h_2, \\
  \frac{1}{2} - \langle \frac{h_2}{N} \rangle & N \mid h_1, n \nmid h_2, \\
0 & \mathrm{otherwise,}
\end{cases}\\
    c_h(n) &:= \sum_{X = \binom{x_1}{x_2} \in L + h, \;
Q(X) = n}
\sgn(x_1), \qquad n > 0.
  \end{split}
\end{equation}
Here $\langle x \rangle \in (0, 1]$ is defined by the property $x - \langle x \rangle  \in \Zb$.
\end{prop}

\begin{proof}
  Use the identity $e^{-2\pi y} = \sqrt{y} \int^\infty_{0} e^{-\pi y (t^2 + t^{-2})} (t + t^{-1}) \tfrac{dt}{t}$ and $H(0, x) = \frac{1}{2} - x$, where $H(s, x) := \sum_{n = 0}^\infty (x + n)^{-s}$ is the Hurwitz zeta function.
\end{proof}

\section{Some Special Functions}
\label{sec:special_fn}
In this section, we will introduce a special function $\tvarphi_\tau$ on $\Rb^{1, 1}$ such that $\xi \tvarphi_\tau = \varphi_\tau$.
\subsection{Non-holomorphic Part}
Define the functions $\fn_\tau : \Rb \to \Rb$ and $\varphi^*_\tau: \Rb^{1, 1} \to \Cb$ by
\begin{equation}
  \label{eq:nonhol_part}
\begin{split}
\fn_\tau(x) &:=   \sgn(x) - \erf \lp \sqrt{2 \pi v} x \rp = \sgn(x) \erfc(\sqrt{2\pi v}|x|), \;
\varphi^*_\tau(x, y) := \ebf \lp \frac{y^2 - x^2}{2} \tau \rp \fn_\tau(x).
\end{split}
\end{equation}
where $\erf(x) := \tfrac{2}{\sqrt{\pi}} \int^x_0 e^{-r^2 }dr$ and $\erfc(x)$ are the error and complementary functions respectively.
Straightforward calculations show that 
\begin{equation}
  \label{eq:xi_image}
  \xi(\varphi^*_\tau(x, y) ) = -\varphi_{\tau}(x, y)
\end{equation}
for all $(x, y) \in \Rb^{1, 1}$.

For each $y \in \Rb$, the function $\varphi^*_\tau(x, y)$ decays like a Schwartz function in $x$.
Also, $\varphi^*_\tau(x, y)$ satisfies 
$$
\lim_{x \to 0^+} \varphi^*_\tau(x, y)  - \lim_{x \to 0^-} \varphi^*_\tau(x, y) = 2 \ebf \lp \frac{y^2}{2} \tau \rp,
$$
hence has a jump discontinuity at $x = 0$.
Away from 0, it is smooth.
Thus, we can view it as a tempered distribution on $\Rb^{1, 1}$ and calculate its Fourier transform with respect to $-Q'$ as follows.

First, notice that as a distribution, $\fn_\tau$ satisfies the differential equation
$$
\frac{d}{dx} \lp \fn_\tau(x)
\rp
=2 \cdot \delta(x) -2\sqrt{2v}  e^{-2 \pi v x^2},
$$
where $\delta(x)$ is the Dirac delta function.
This follows from $\tfrac{d}{dx} |x| = \sgn(x), \tfrac{d}{dx} \sgn(x) = 2 \delta(x)$ as tempered distributions.
Substituting in the definition of $\varphi^*_\tau(x, y)$, we see that it satisfies
\begin{equation}
\label{eq:diff_eq_1}
\frac{\partial }{ \partial x}\lp \varphi^*_\tau(x, y) \rp + 2\pi i {\tau} x  \varphi^*_\tau(x, y) 
= 
\lp 2 \delta(x)
-
2\sqrt{2v}  \ebf \lp- \frac{x^2}{2} \overline{\tau} \rp \rp \ebf \lp \frac{y^2}{2} \tau \rp.
\end{equation}
Notice that $\ebf \lp - \tfrac{x^2}{2} {\tau}  \rp 2 \delta(x) = 2\delta(x)$ as a distribution.

For a Schwartz function $\phi$ on $\Rb^{1, 1}$, we define its Fourier transform $\Fc(\phi)$ with respect to the quadratic form $-Q'$ by
\begin{equation}
  \label{eq:FT}
  \Fc(\phi)(x, y) := \int_{\Rb^{1, 1}} \phi(w, z) \ebf(-wx + yz) dw dz.
\end{equation}
If $\phi$ is not a Schwartz function and the integral above converges, we also use it to denote its Fourier transform.
Using the standard facts of Fourier transform (see e.g.\ \cite[Lemma 3.1]{Borcherds98}), we have
$$
- {\tau}  \frac{\partial}{\partial \hx} \Fc(\varphi^*_\tau)(\hx, \hy) 
+ 2\pi i \hx \Fc(\varphi^*_\tau)(\hx, \hy) 
= 
\lp 2 - 2 \sqrt{2 v} \lp i\overline{\tau} \rp^{-1/2}
 \ebf \lp - \frac{\hx^2}{2} \overline{ \lp - 1/\tau \rp} 
 \rp \rp
\frac{ \ebf \lp \frac{y^2}{2} (-1/\tau) \rp}{\sqrt{-i \tau}}.
$$
After dividing by $-\tau$ on both sides and making the change of variable $\tau \mapsto -1/\tau$, the equation becomes
\begin{align*}
 \frac{\partial}{ \partial \hx} \Fc(\varphi^*_{-1/\tau})(\hx, \hy)   &+
2\pi i \hx \tau \Fc(\varphi^*_{-1/\tau})\lp \hx, \hy \rp 
= 
-  {2\tau} \lp \sqrt{2v }  \ebf \lp - \frac{\hx^2}{2} \overline{\tau }  \rp
- \sqrt{-i\tau} \rp \ebf \lp \frac{\hy^2}{2} \tau \rp,
\end{align*}

Now define
\begin{equation}
  \label{eq:Dn_diff}
  \Dn_\tau(x, y) :=  \varphi^*_\tau(x, y) - \frac{\Fc(\varphi^*_{-1/\tau})(x, y)}{\tau} .
\end{equation}
Then it satisfies the differential equation 
\begin{equation}
  \label{eq:D_diff}
 \ebf \lp - \frac{x^2}{2} \tau  \rp \frac{d}{dx} 
\lp \ebf \lp \frac{x^2}{2} \tau  \rp \Dn_\tau(x, y) \rp  
= 
2 \lp \delta(x) - \sqrt{-i\tau} \rp
\ebf \lp \frac{y^2}{2} \tau \rp.
\end{equation}
We have the following result concerning the solutions to this differential equation.
\begin{prop}
  \label{prop:diff_Dn}
For fixed $\tau_0 \in \Hc, y_0 \in \Rb$, the only jump discontinuity of any piecewise continuous solution to the differential equation \eqref{eq:D_diff} is at $x = 0$. 
Suppose further that it is bounded in $x$. 
Then the solution agrees with the function $\Dc_{\tau_0}(x,y_0)$ defined by
\begin{equation}
  \label{eq:Dc}
  \Dc_{\tau_0}(x, y_0) := 
 \ebf \lp \frac{y_0^2 - x^2}{2} \tau_0 \rp 
\sgn(x) \erfc(\sqrt{-i \tau_0} |x|)
\end{equation}
whenever the solution is continuous.
In particular, $\Dn_\tau(x, y) = \Dc_\tau(x, y)$ for all $(x, y) \in \Rb^{1, 1}$ and $\tau \in \Hc$. 
\end{prop}

\begin{rmk}
  Here in $\tau_0 = u_0 + iv_0$, the function $\erfc(\sqrt{-i\tau_0}|x|)$ is the unique holomorphic extension of $\erfc(\sqrt{v_0}|x|)$. 
\end{rmk}

\begin{proof}
The first claim is clear as a jump discontinuity at $x = x_0$ of a piecewise solution would produce a constant times $\delta(x - x_0)$.
By the fundamental theorem of calculus, the solution, whenever continuous, would agree with $\Dc_{\tau_0}(x,y_0)$ up to a constant multiple of $\ebf \lp - \frac{x^2}{2} \tau \rp$, which is unbounded as $x \to \infty$.
Since $\Dc_{\tau_0}(x ,y_0)$ is assumed to be bounded, the second claim follows.
Finally, for any fixed $\tau_0 \in \Hc, y_0 \in \Rb$, we have $\varphi^*_{\tau_0}(x, y_0) \in L^1(\Rb)$. 
Thus its Fourier transform is continuous and bounded. 
That implies $\Dn_\tau(x, z)$ is bounded and has no removable discontinuity on $\Rb$, hence the third claim.
\end{proof}

\subsection{Holomorphic Part}
Now, we will define the holomorphic counterpart to $\varphi^*_\tau$ as
\begin{equation}
  \label{eq:hol_part}
  \varphi^+_\tau(x, y) := 
\ebf \lp \frac{y^2 - x^2}{2} \tau \rp \sgn(x) \mathds{1}_{y^2 > x^2},
\end{equation}
where $\mathds{1}_{y^2 > x^2}$ is the characteristic function of the set $\{(x, y) \in \Rb^{1, 1}: y^2 > x^2 \}$.
Even though $\varphi^+_\tau(x, y)$ is not a Schwartz function, it decays nicely enough such that we have the following result.
\begin{prop}
  \label{prop:FT_converge}
The following integral 
$$
\Fc(\varphi^+_\tau)(\hx, \hy) := \int_{\Rb^{1, 1}} \varphi^+_\tau(w, z) \ebf(- w \hx + z \hy) dw dz
$$
converges uniformly on compact subsets of $\{(\hx, \hy) \in \Rb^{1, 1}: \hx^2 \neq \hy^2 \}$.
Furthermore, the function $\Fc(\varphi^+_\tau)$ is bounded and continuously differentiable on $\{(\hx, \hy) \in \Rb^{1, 1}: \hx^2 \neq \hy^2 \}$.
\end{prop}

\begin{proof}
Let $A = \tfrac{1}{\sqrt{2}} \smat{1}{1}{-1}{1}$ and make the rotational change of variables
$$
\binom{a}{b} := A \cdot \binom{w}{z}, \;
\binom{x'}{y'} := A \cdot \binom{x}{y},
$$
we can rewrite the integral above as
$$
\Fc(\varphi^+_\tau)(x', y') = \lim_{T, T' \to \infty} \int_{-T'}^{T'} \int^{T}_{-T} \mathds{1}_{ab > 0} \ebf(ab \tau) \sgn(a - b) \ebf(ay' + bx') db da.
$$
The integral over $0 < a < T', 0 < b < T$ can be evaluated explicitly as
$$
\int^{T'}_0 \frac{\ebf(a^2 \tau + a(x' + y'))}{\pi i (a \tau + x')} - \frac{\ebf((a\tau + x')T + ay')}{2\pi i (a \tau + x')} - \frac{\ebf(a y')}{2\pi i (a \tau + x')} da.
$$
The same can be done for the region $-T' < a < 0, -T < b < 0$.
As $T \to \infty$, the middle term vanishes, and we are left with
\begin{equation}
\label{eq:FT_lim}
\Fc(\varphi^+_\tau)(x', y') = 
\lim_{T' \to \infty} \int^{T'}_{-T'}   \frac{\ebf(a^2 \tau + a(x' + y'))}{\pi i (a \tau + x')} - \frac{\ebf(a y')}{2\pi i (a \tau + x')} da.
\end{equation}
The integral of the first term can be bounded with $\int^\infty_{-\infty} \frac{e^{-r^2}}{\sqrt{r^2 + (x')^2}} dr \ll |x'|^{-1}$, which implies that the integral converges uniformly and defines a continuously differentiable function away from $x'y' = 0$.
Furthermore, it is bounded when $|x'|$ is large. 
On the other hand when $|x'|$ is close to zero, we can fix an absolute constant $\epsilon > 0$ such that the integral over $|a| \in (\epsilon , \infty)$ converges absolutely independent of $x', y'$. 
The rest of the integrand can be written as 
$$
\frac{\ebf( a^2 \tau + a(x' + y'))}{a \tau + x'}  + \frac{\ebf(a^2 \tau - a(x' + y'))}{- a \tau + x'}  
=
C_1(a, x', \tau) \frac{x'}{(a\tau)^2 - (x')^2} 
+ 
C_2(a, x', \tau) \frac{\sin(ay')}{a} 
$$
with $|C_j(a, x', \tau)|$ bounded above independently of $a$ and $x'$. Since $|\int^\epsilon_0 \frac{x'}{(a\tau)^2 - (x')^2} da |$ and $|\int^\epsilon_0 \frac{\sin(ay')}{a} da  |$ are bounded independent of $x'$ and $y'$, the integral of the first term in Eq.\ \eqref{eq:FT_lim} defines a bounded and continuously differentiable function on $\{(x, y) \in \Rb^{1, 1}: x^2 \neq y^2\}$.

Away from $x'y' = 0$, the integral of the last term converges uniformly using integration by parts and defines a continuous function. 
Using standard formula in one dimensional Fourier transform, we can in fact evaluate it explicitly as
$$
\int^\infty_{-\infty} \frac{\ebf(a y')}{2\pi i (a \tau + x')} da
= \tau^{-1}  \ebf(x'y' (-1/\tau)) \mathds{1}_{x'y' > 0}.
$$
From this, it is clear that it is bounded.
\end{proof}

Since $\varphi^+_\tau$ is bounded, integrating against it defines a tempered distribution on $\Rb^{1, 1}$.
Thus, we can then study its Fourier transform $\Fc(\varphi^+_\tau)$ as we have done for $\varphi^*_\tau$. 
The analogue to equation \eqref{eq:diff_eq_1} is as follows
\begin{equation}
\label{eq:diff_eq_2}
\frac{\partial  \varphi^+_\tau }{ \partial x}
 + 2\pi i {\tau} x  \varphi^+_\tau
= 
\lp 2 \delta(x) - \delta(x - y) - \delta(x + y)  \rp \ebf \lp \frac{y^2 - x^2}{2} \tau \rp.
\end{equation}
Applying Fourier transform to both sides and making the change $\tau \mapsto -1/\tau$
 yields
$$
(-1/\tau)  \frac{\partial  \Fc(\varphi^+_{-1/\tau})( \hx, \hy)}{\partial \hx}
- 2\pi i \hx \Fc(\varphi^+_{-1/\tau}) (\hx, \hy)
= 
-\lp 2 {\sqrt{-i \tau}}
{ \ebf \lp \frac{\hy^2}{2} \tau \rp}
- \delta(\hy - \hx) - \delta(\hx + \hy)\rp
$$
Subtracting the previous two equations shows that the function defined by
\begin{equation}
  \label{eq:Dh_diff}
  \Dh_\tau(x, y) :=  \varphi^+_\tau(x, y) - \frac{\Fc(\varphi^+_{-1/\tau})(x, y)}{\tau} 
\end{equation}
also satisfies the differential equation \eqref{eq:D_diff}.
Note that $\delta(\hy \pm \hx) = \delta(\hy \pm \hx) \ebf ( \tfrac{\hy^2 - \hx^2}{2} )$ and $\delta(x) \ebf(\tfrac{x^2}{2} \tau) = \delta(x)$.
For each fixed $\tau_0 \in \Hc$, the function $\varphi^+_{\tau_0}$ is bounded with only jump singularities when either $x^2 = y^2$ or $x = 0$.
Proposition \ref{prop:FT_converge} implies that $\Fc(\varphi^+_\tau)$ has the same property as well.
So we can define $\Fc(\varphi_\tau^+)(x, y)$ on $x^2 = y^2$ such that $\Dc_\tau^+(x, y)$ is continuous when $y^2 = x^2 > 0$.
By Proposition \ref{prop:diff_Dn}, we know that $\Dh_\tau = \Dc_{\tau}$ and have proved the following result.
\begin{prop}
\label{prop:tvarphi}
For all $(x, y) \in \Rb^{1, 1}$, the $L^\infty(\Rb^{1, 1})$ function 
\begin{equation}
  \label{eq:varphi_mod}
  \tvarphi_\tau(x, y) := \varphi^+_\tau(x, y) - \varphi^*_\tau(x, y) = \sgn(x) \ebf \lp \frac{y^2 - x^2}{2} \rp \lp \mathds{1}_{y^2 > x^2} -\erfc(\sqrt{2\pi v} |x|)\rp
\end{equation}
satisfies 
\begin{enumerate}
\item $ \tvarphi_{\tau + 1}(x, y) = \ebf ( \tfrac{y^2 - x^2}{2}) \tvarphi_\tau(x, y)$,
\item $ \Fc(\tvarphi_{-1/\tau})(x, y) =  \tau \tvarphi_{\tau}(x, y)$,
\item $ \xi( \tvarphi_\tau(x, y)) = \varphi_\tau(x, y)$.
\end{enumerate}

\end{prop}

\section{Real-Analytic Theta Series}
In this section, we will construct weight 1 real-analytic theta series $\tvartheta(\tau)$ that transforms with respect to $\rho_{-L}$ and maps to $\vartheta(\tau)$ under $\xi$. 
Proposition \ref{prop:tvarphi} and the construction of $\vartheta_h(\tau)$ imply that we need to consider summing $\tvarphi_\tau(\iota_t(X))$ over $X \in L+ h$ and integrating over $\Rb^\times_+$ with respect to $\frac{dt}{t}$. 
However, the sum and integral are both divergent.
The problem with the sum is caused by isotropic elements in $L$.
We will regularize the sum by considering a slight shift of the lattice $L$, and regularize the integral by adding the converging factor $t^s$ as usual. 
The ideas are simple, but the procedure to carry it out is a bit complicated.

In the notations of sections \ref{sec:theta_lift} and \ref{sec:special_fn}, define the following series
\begin{equation}
  \label{eq:tTheta_varep}
  \begin{split}
\tTheta_h(\tau, t; \varepsilon, \varepsilon') 
&:=  \sum_{X \in L + h + \varepsilon \binom{1}{1}} \tvarphi_\tau \lp \iota_t(X) \rp
\ebf \lp B \lp X, \varepsilon' \binom{1}{1} \rp \rp
      \end{split}
\end{equation}
for $\varepsilon, \varepsilon' \in (-\half, \half)$. 
This series converges for $\varep, \varep' \in (-\half, \half)$. It is modular for $\varep, \varep' \in (-\half, \half) \backslash \{0\}$, but not continuous at $\varep = 0$. 
Define $\Theta^*_h$ and $\Theta^+_h$ as $\tTheta_h$ in \eqref{eq:tTheta_varep} with $\tvarphi_\tau$ replaced by $\varphi_\tau^*$ and $\varphi^+_\tau$ respectively.
To preserve the modularity, we define the theta series $\tTheta(\tau, t) = \sum_{h \in L^*/L} \tTheta_h(\tau, t) \ef_h$ by
\begin{equation}
  \label{eq:tTheta}
  \tTheta_h(\tau, t) := c_h(0) + \tTheta_h(\tau, t; 0, 0) = c_h(0) + \Theta^+_h(\tau, t; 0, 0) + \Theta^*_h(\tau, t; 0, 0),
\end{equation}
where $c_h(0)$ is defined in Proposition \ref{prop:FE}.
They have the following relationship.
\begin{prop}
  \label{prop:tTheta_converge}
For fixed $\tau \in \Hc$, $t \in \Rb^\times_+$ and $h = \binom{h_1}{h_2} \in \Zb^2$, the series $\tTheta_h(\tau, t; \varepsilon, \varepsilon')$ converges uniformly for $(\varepsilon, \varepsilon')$ in compact subsets of $(-\half, \half) \backslash \{ 0\} \times (-\half, \half)$. 
It is continuous for $(\varep, \varep') \in (0, \min\{\tfrac{1}{1+t^2}, \tfrac{t^2}{1+t^2} \}) \times (-\half, \half)$ and satisfies
\begin{equation}
  \label{eq:tTheta_limit}
  \lim_{\varep \to 0^+} \tTheta_h(\tau, t; \varep, \pm \varep) - \frac{c_{-1}(h)}{2 \pi i (\tau \mp 1) \varep} = \tTheta_h(\tau, t),
\end{equation}
where $c_{-1}(h)\in \{0, 1, 2\}$ is the number of $h_1, h_2$ that are divisible by $N$.
\end{prop}

\begin{proof}
Since $\varphi^*_\tau$ decays like a Schwartz function, the series $\Theta_h^*$ converges absolutely and uniformly for $\varep, \varep' \in \Rb$, except for $h = \binom{0}{0}$, in which case, we have $\lim_{\varep \to 0^+} \Theta_h^*(\tau, t; \varep, \pm \varep) = \Theta_h^*(\tau, t; 0, 0) + 1$. 
For $\Theta^+_h$, notice that $B(X, Z_t)^2 - B(X, W_t)^2 = -\frac{2 x_1 x_2}{N}$ if $X = \binom{x_1}{x_2} \in V_\Rb$. 
So we can write
\begin{equation}
\label{eq:Theta_plus}
\begin{split}
\Theta^+_h(\tau, t; \varep, \varep') & = 
\sum_{\begin{subarray}{c} 
n_1 \in N\Zb + h_1  \\
n_2 \in N\Zb + h_2  \\
n_1 n_2 \le -1 \end{subarray}} 
+
\sum_{\begin{subarray}{c} 
n_1 \in N\Zb + h_1  \\
n_2 \in N\Zb + h_2  \\
(n_1 + \varep) (n_2 + \varep) < 0 \\
n_1 n_2 = 0 \end{subarray}} \\
& \sgn ( t^{-1}(n_1 + \varep) + t(n_2 + \varep))
\ebf \lp - \frac{(n_1 + \varep)(n_2 + \varep)}{N} \tau \rp 
\ebf \lp \frac{(n_1 + n_2 + 2\varep)}{N} \varep'\rp.
\end{split}
\end{equation}
Using the inequality $-(n_1 + \varep )(n_2 + \varep) > -\frac{n_1n_2}{2}$ for $\varep \in (-\half, \half)$, we see that the first sum in equation \eqref{eq:Theta_plus} converges absolutely and uniformly for $(\varep, \varep')$ in compact subset of $(-\half, \half)^2$.
Note that the second sum is empty if and only if $N \nmid h_j$ for $j = 1, 2$.
Suppose $N \mid h_1$ and $\varep > 0$.
Then $n_2 \le -1$ and the summand becomes 
$$
\sgn ( t^{-1} \varep + t(n_2 + \varep))
\ebf \lp  \frac{\varep \tau - \varep' }{N} (-n_2) \rp 
\ebf \lp  \frac{2\varep \varep' - \varep^2 \tau}{N} \rp
$$
Since $t > 0$, $\varep < \min\{ \tfrac{1}{1+t^2}, \tfrac{t^2}{1+t^2} \}$ and $n_2 \le -1$ , we have $t^{-1} \varep + t(n_2 + \varep) < t^{-1} \varep + t(-1 + \varep) < 0$.
Then the second sum is just a geometric series and equals to
$$
- \ebf \lp \frac{2\varep \varep' - \varep^2 \tau}{N}  \rp 
\frac{\ebf\lp  (  \varep \tau -   \varep' ) \langle - \frac{h_2}{N} \rangle \rp}{1 - \ebf \lp \varep \tau -     \varep' \rp}.
$$
Using the power series expansion $- \frac{e^{ax}}{1 - e^x} =  x^{-1} - (\half - a) + O(x)$, we see that we get a constant term $- \half + \langle - \tfrac{h_2}{N} \rangle$ when $\varep' = \pm \varep$.
When $h \neq \binom{0}{0}$, this is just $c_h(0)$. 
When $h = \binom{0}{0}$, there are twice this contribution, which sums to $c_h(0) + 1$.
Now we are done since $\tTheta_h(\tau, t; \varep, \varep') = \Theta^+_h(\tau, t; \varep, \varep') - \Theta^*_h(\tau, t; \varep, \varep')$.
\end{proof}

\begin{prop}
\label{prop:tTheta_mod}
The theta function $\tTheta(\tau, t)$ is a real-analytic modular form in $\tau$ of weight 1 with respect to $\rho_{-L}$ and satisfies $\xi(\tTheta(\tau, t)) = \Theta(\tau, t)$ and $\tTheta_h(\tau, t) = O_\tau(1)$ for all $t \in \Rb^\times_+$.
\end{prop}

\begin{proof}
The property $\xi(\tTheta_h(\tau, t)) = \Theta_h(\tau, t)$ and the modularity in $T$ are clear from the definition. For the modularity in $S$, we can apply Poisson summation to obtain 
\begin{equation}
  \label{eq:tTheta_varep_mod}
  \frac{\tTheta_h(-1/\tau, t; \varepsilon, \varep')}{\tau} =
\frac{\ebf(2\varep \varep')}{N}
 \sum_{\delta \in L^*/L} \ebf( (\delta, h))
 \tTheta_\delta(\tau, t; -\varep', \varep)
\end{equation}
with $\varep, \varep' \in (-\half, \half) \backslash \{0\}$.
Using the identity $c_{-1}(h) = \frac{\sum_{\delta \in L^*/L} \ebf((\delta, h)) c_{-1}(\delta)}{N}$, we obtain the desired modularity with respect to $S$ after setting $\varep' = -\varep$, subtracting $\frac{c_{-1}(h)}{2 \pi i  (\tau - 1) \varep}$ from both sides and taking the limit $\varep \to 0^+$. 
The asymptotic of $\tTheta(\tau, t)$ in $t$ can be seen from its definition, the decay of $\varphi^*_\tau$, and the expression \ref{eq:Theta_plus}.
\end{proof}

Now to construct the preimage of $\vartheta_h(\tau)$ under $\xi$, we consider the integral 
\begin{equation}
\label{eq:tvarthetas}
\tvartheta_h(\tau; s) := \int^\infty_1 \tTheta_h(\tau, t) t^{-s} \frac{dt}{t}
+ \int^1_0 \tTheta_h(\tau, t) t^{s} \frac{dt}{t},
\end{equation}
which converges for $\Re(s) > 0$ and can be analytically continued to $s \in \Cb$ via its Fourier expansion in $\tau$. 
We are interested in the function 
\begin{equation}
  \label{eq:tvartheta}
  \tvartheta_h(\tau) := \mathrm{Const}_{s = 0} \tvartheta_h(\tau; s).
\end{equation}
It has the following desirable properties.
\begin{thm}
  \label{thm:tvartheta}
The function $\tvartheta(\tau) := \sum_{h \in L^*/L} \tvartheta_h(\tau) \ef_h$ is a harmonic Maass form of weight 1 with respect to $\rho_{-L}$, and maps to the Eisenstein series $\vartheta(\tau)$. 
It has the Fourier expansion
$  \tvartheta_h(\tau) = \sum_{n \in \Qb_{\ge 0}} \tc_h(n) q^n + c_h(0) \log v - \sum_{n \in \Qb_{> 0}} c_h(n) \Gamma(0, 4\pi v n) q^{-n}
$, where $c_h(n) \in \Qb$ are defined in Eq.\ \eqref{eq:FC} and $\tc_h(n)$ are defined by
\begin{equation}
  \label{eq:tvartheta_hol_FC}
  \begin{split}
    \tc_h(0) & := 
    \begin{cases}
      c_h(0) \lp \log (\pi N) - \frac{\Gamma'(1/2)}{\Gamma(1/2)} \rp - 
\log \Gamma \lp \langle \frac{h_1}{N} \rangle \rp + \log \Gamma \lp - \langle \frac{h_1}{N} \rangle \rp, &N \mid h_2, \\
      c_h(0) \lp \log (\pi N) - \frac{\Gamma'(1/2)}{\Gamma(1/2)} \rp - 
\log \Gamma \lp \langle \frac{h_2}{N} \rangle \rp + \log \Gamma \lp - \langle \frac{h_2}{N} \rangle \rp, & N \mid h_1, \\
0, & \text{otherwise,}
    \end{cases}
\\
\tc_h(n) &:= \sum_{X = \binom{x_1}{x_2} \in L + h, \; -Q(X) = n} \sgn(x_1) \log \left| \frac{x_1}{x_2} \right|, \quad n > 0.
  \end{split}
\end{equation}
\end{thm}

\begin{proof}
The modularity statement follows from Prop.\ \ref{prop:tTheta_mod}.
For the Fourier expansion, we will first calculate the contribution of $\Theta^*_h$ in the integral defining $\tvartheta_h(\tau, s)$, i.e.
$$
\vartheta^*_h(\tau, s) := \int^\infty_1 \Theta^*_h(\tau, t; 0, 0) t^{-s} \frac{dt}{t}
+ \int^1_0 \Theta^*_h(\tau, t; 0, 0) t^{s} \frac{dt}{t}.
$$
Since the sum defining $\Theta^*_h(\tau, t; 0, 0)$ converges absolutely and uniformly in $t$, we can switch the sum and integral. It is then suffices to consider the integral 
\begin{equation}
\label{eq:I*h}
I_h^*(X, \tau, s) := \int^\infty_1 \varphi^*_{\tau}(\iota_t(X)) t^{-s} \frac{dt}{t} + \int^1_0 \varphi^*_{\tau}(\iota_t(X)) t^{s} \frac{dt}{t}
\end{equation}
for each $X = \binom{x_1}{x_2} \in L + h$. 
If $n := Q(X) = \frac{x_1x_2}{N} \neq 0$, then $\mathrm{Const}_{s = 0} I^*_h(X, \tau, s) = I^*_h(X, \tau, 0)$ and
\begin{align*}
  I^*_h(X, \tau, 0) &= \int^\infty_0 \varphi^*_\tau(\iota_t(X))  \frac{dt}{t}
= \ebf \lp - n \tau \rp \int^\infty_0 \sgn(x_1 t^{-1} + x_2 t) \erfc \lp \sqrt{\frac{\pi v}{N}} |x_1 t^{-1} + x_2 t| \rp \frac{dt}{t} \\
&= \sgn(x_1)\ebf \lp - n \tau \rp \int^\infty_0 \sgn(w^{-1} + \sgn(n) w) \erfc \lp \sqrt{n \pi v} | w^{-1} + \sgn(n) w| \rp \frac{dw}{w},
\end{align*}
where $w = {\left| \frac{x_2}{x_1} \right|}^{1/2} t$.
For any $\alpha > 0$, we have $\int^\infty_0 \sgn(w^{-1} - w) \erfc(\alpha|w^{-1}  - w|) \frac{dw}{w} = 0$ and 
$$
\int^\infty_0  \erfc(\alpha(w^{-1}  + w)) \frac{dw}{w} = \Gamma(0, 4\alpha^2),
$$
with $\Gamma(s, x) := \int^\infty_x t^{s-1} e^{-t} dt$ the incomplete gamma function.
Therefore, we have
$$
\mathrm{Const}_{s = 0} I^*_h(X, \tau, s) = \sgn(x_1) \ebf(-Q(X) \tau) \Gamma(0, 4\pi Q(X) v)
$$
when $X$ is not isotropic.

When $Q(X) = 0$, we know that $I^*_h(X, \tau, s) = 0$ if $x_1 = x_2 = 0$. So suppose $x_1 \neq 0$ and $x_2 = 0$. 
Simple estimate shows that 
$$
\left| I^*_h\lp  \binom{x_1}{0}, \tau, s \rp -
\sgn(x_1 )\int^\infty_0  \erfc \lp \sqrt{\frac{\pi v}{N}} |x_1 t^{-1}| \rp t^{-s} \frac{dt}{t} \right| \ll e^{-c x_1^2} \cdot \int^\infty_1 e^{-c (t^2-1)} |t^s - t^{-s}| \frac{dt}{t},
$$
where $c = \tfrac{\pi v }{N}$. 
Using the formula $\int^\infty_0 \erfc(\alpha t) t^s \frac{dt}{t} = \frac{\alpha^{-s}}{\sqrt{\pi} s} \Gamma \lp \frac{s+1}{2} \rp$, we have
\begin{align*}
 \sum_{\begin{subarray}{c} x_1 \equiv h_1 (N)\\ x_1 \neq 0 \end{subarray}} I^*_h \lp  \binom{x_1}{0}, \tau, s \rp
&=  \frac{(N \pi v)^{-s/2}}{\sqrt{\pi} s} \Gamma \lp \frac{s+1}{2} \rp
\lp H \lp s, \langle \frac{h_1}{N} \rangle \rp -  H \lp s, \langle - \frac{h_1}{N} \rangle \rp \rp.
\end{align*}
The constant term at $s = 0$ of the right hand side is then given by 
$-c_h(0) \log v - \tc_h(0)$.

Now, we will consider the following integral 
$$
\vartheta^+_h(\tau, s) := \int^\infty_1 \Theta^+_h(\tau, t; 0, 0) t^{-s} \frac{dt}{t}
+ \int^1_0 \Theta^+_h(\tau, t; 0, 0) t^{s} \frac{dt}{t}.
$$
By the definition of $\varphi^+_\tau$, it suffices to calculate as before
\begin{equation}
\label{eq:I+h}
I_h^+(X, \tau, s) := \int^\infty_1 \varphi^+_{\tau}(\iota_t(X)) t^{-s} \frac{dt}{t} + \int^1_0 \varphi^+_{\tau}(\iota_t(X)) t^{s} \frac{dt}{t}
\end{equation}
for $X \in L + h$ with $- Q(X) > 0$. 
For $X = \binom{x_1}{x_2}$, this simplifies to
\begin{align*}
I^+_h(X, \tau, s) &= \ebf(-Q(X) \tau) \lp \int^\infty_1 \sgn(t^{-1} x_1 + t x_2) t^{-s} \frac{dt}{t} + \int^1_0 \sgn(t^{-1} x_1 + t x_2) t^s \frac{dt}{t} \rp \\
&= \sgn(x_1) \ebf(-Q(X) \tau) \lp r^s \int^\infty_r \sgn(w^{-1} - w) w^{-s} \frac{dw}{w} + r^{-s} \int^r_0 \sgn(w^{-1} - w) w^s \frac{dw}{w} \rp 
\end{align*}
after a change of variable $w = r \cdot t, r = |x_2/x_1|^{1/2}$.
For $\Re(s) > 0$, we then have $I^+_h(X, \tau, s) = \sgn(x_1) \ebf(-Q(X)\tau) \frac{2(r^{-s} - 1)}{s}$. Taking the limit as $s$ goes to 0 then finishes the calculation.
\end{proof}

\bibliography{HEisen}{}

\providecommand{\bysame}{\leavevmode\hbox to3em{\hrulefill}\thinspace}
\providecommand{\MR}{\relax\ifhmode\unskip\space\fi MR }
\providecommand{\MRhref}[2]{%
  \href{http://www.ams.org/mathscinet-getitem?mr=#1}{#2}
}
\providecommand{\href}[2]{#2}
\begin{thebibliography}{10}

\bibitem{Borcherds98}
Richard~E. Borcherds, \emph{Automorphic forms with singularities on
  {G}rassmannians}, Invent. Math. \textbf{132} (1998), no.~3, 491--562.

\bibitem{BO10}
Jan Bruinier and Ken Ono, \emph{Heegner divisors, {$L$}-functions and harmonic
  weak {M}aass forms}, Ann. of Math. (2) \textbf{172} (2010), no.~3,
  2135--2181.

\bibitem{BFI15}
Jan~H. Bruinier, Jens Funke, and {\"O}zlem Imamo{\=g}lu, \emph{Regularized
  theta liftings and periods of modular functions}, J. Reine Angew. Math.
  \textbf{703} (2015), 43--93.

\bibitem{BF04}
Jan~Hendrik Bruinier and Jens Funke, \emph{On two geometric theta lifts}, Duke
  Math. J. \textbf{125} (2004), no.~1, 45--90.

\bibitem{DIT11}
W.~Duke, {\"O}.~Imamo{\=g}lu, and {\'A}.~T{\'o}th, \emph{Cycle integrals of the
  {$j$}-function and mock modular forms}, Ann. of Math. (2) \textbf{173}
  (2011), no.~2, 947--981.

\bibitem{Hecke26}
E.~Hecke, \emph{Zur {T}heorie der elliptischen {M}odulfunktionen}, Math. Ann.
  \textbf{97} (1927), no.~1, 210--242.

\bibitem{Kudla81}
Stephen~S. Kudla, \emph{Holomorphic {S}iegel modular forms associated to {${\rm
  SO}(n,\,1)$}}, Math. Ann. \textbf{256} (1981), no.~4, 517--534.

\bibitem{Kudla97}
\bysame, \emph{Central derivatives of {E}isenstein series and height pairings},
  Ann. of Math. (2) \textbf{146} (1997), no.~3, 545--646.

\bibitem{KM86}
Stephen~S. Kudla and John~J. Millson, \emph{The theta correspondence and
  harmonic forms. {I}}, Math. Ann. \textbf{274} (1986), no.~3, 353--378.

\bibitem{KM87}
\bysame, \emph{The theta correspondence and harmonic forms. {II}}, Math. Ann.
  \textbf{277} (1987), no.~2, 267--314.

\bibitem{KRY99}
Stephen~S. Kudla, Michael Rapoport, and Tonghai Yang, \emph{On the derivative
  of an {E}isenstein series of weight one}, Internat. Math. Res. Notices
  (1999), no.~7, 347--385.

\bibitem{Siegel51}
Carl~Ludwig Siegel, \emph{Indefinite quadratische {F}ormen und
  {F}unktionentheorie. {I}}, Math. Ann. \textbf{124} (1951), 17--54.

\bibitem{Weil65}
Andr{\'e} Weil, \emph{Sur la formule de {S}iegel dans la th\'eorie des groupes
  classiques}, Acta Math. \textbf{113} (1965), 1--87.

\end{thebibliography}
\bibliographystyle{amsplain}

\end{document}